\documentclass[english]{amsart}

\usepackage{esint}
\usepackage[svgnames]{xcolor} 
\usepackage[colorlinks,citecolor=red,pagebackref,hypertexnames=false,breaklinks]{hyperref}
\usepackage{pgf,tikz}
\usepackage{pdfsync}

\usepackage{dsfont}
\usepackage{url}
\usepackage[utf8]{inputenc}
\usepackage[T1]{fontenc}
\usepackage{lmodern}
\usepackage{babel}
\usepackage{mathtools}  
\usepackage{amssymb}
\usepackage{lipsum}
\usepackage{mathrsfs}
\usepackage{color}
\usepackage[skip=2.2pt plus 1pt, indent=12pt]{parskip}
\usepackage{stmaryrd}
\usepackage{soul}

\newtheorem{theorem}{Theorem}[section]
\newtheorem{proposition}{Proposition}[section]
\newtheorem{lemma}{Lemma}[section]

\newtheorem{remark}{Remark}[section]

\numberwithin{equation}{section}

\title[Upper bound on the multiplicity of eigenvalues]{Upper bound on the multiplicity of eigenvalues of the Sch\"odinger-Dirichlet operator in dimension two}

\author{Mourad Choulli}
\address{Universit\'{e} de Lorraine, 34 cours L\'{e}opold, 54052 Nancy cedex, France}
\email{mourad.choulli@univ-lorraine.fr}

\date{}

\begin{document}

\begin{abstract}
We establish an upper bound on the multiplicity of eigenvalues of the Sch\"odinger-Dirichlet operator in dimension two. We give a proof based on a generalized Morse Lemma due to Cheng \cite{Ch}.
\end{abstract}

\subjclass[2010]{35B05, 35J15, 35P15}

\keywords{Schr\"odinger-Dirichlet operator, multiplicity of the eigenvalues, nodal domain.}

\maketitle

\section{Introduction}

All the functions we will consider are assumed to have real values.

Let  $n\ge 2$ be an integer, $\Omega$ a bounded domain of $\mathbb{R}^n$ and $V\in L^\infty (\Omega)$. We denote the sequence of eigenvalues of the operator 
\[
A:=-\Delta+V,\quad D(A)=\{f\in H_0^1(\Omega);\; \Delta u\in L^2(\Omega)\}
\] 
by $(\lambda_k)_{k\ge 1}$. These eigenvalues are counted according to their multiplicity and arranged as a non-decreasing sequence:
\[
-\infty <\lambda_1<\lambda_2\le \lambda_3\le \ldots \le \lambda_k\le \ldots \quad \mathrm{and}\quad \lim_{k\rightarrow \infty}\lambda_k=\infty.
\]

Let $k\ge 2$. If  there exists an integer $p\ge 0$ such that
\[
\lambda_{k-1}<\lambda_k=\ldots =\lambda_{k+p}<\lambda_{k+p+1},
\]
then $m_k:=p+1$ is  the multiplicity of $\lambda_k$. It is the dimension of the eigenspace $E_k$ associated with the eigenvalue $\lambda_k$.

\begin{theorem}\label{mthm1}
Let $n=2$ and assume that $\Omega$ is of class $C^\infty$ and $V\in C^\infty (\overline{\Omega})$. For all $k\ge 2$, we have $m_k\le 2k-1$. 
\end{theorem}

By the usual elliptic regularity, when the assumptions of Theorem \ref{mthm1} are verified, we have $E_k\subset C^\infty(\overline{\Omega})$ for all $k\ge 1$.

Theorem \ref{mthm1} can be improved by assuming an additional symmetry condition, as shown by the following result. In what follows, $\varrho :(x,y)\in \mathbb{R}^2\mapsto (y,x)\in \mathbb{R}^2$.

\begin{theorem}\label{mthm3}
In addition to the assumptions of Theorem \ref{mthm1}, suppose that $\varrho(\Omega)=\Omega$ and $V=V\circ \varrho$. For all $k\ge 2$, we have $m_k\le 2k-2$.
\end{theorem}

In the case $V=0$, with the assumptions of Theorem \ref{mthm1}, we have $m_k\le 2k-3$ for all $k\ge 3$ when $\Omega$ is simply connected. This result was initially proven by Hoffmann-Ostenhof, Michor and Nadirashvili \cite{HMN} without assuming that $\Omega$ is simply connected. As pointed out by Berdnikov \cite{Be}, the proof in \cite{HMN} contains a gap. Berdnikov also shows in \cite{Be} that the bound $m_k\le 2k-3$ holds for all $k\ge 3$ if $\Omega$ is assumed to be simply connected.

In this short article, we give a simple proof of Theorem \ref{mthm1}. Although some of the results we use are more or less known, we provide their detailed proof for the sake of completeness. The proof of the estimate $m_k\le 2k-3$ in the case where $\Omega$ is simply connected and $V=0$ requires a more thorough analysis of the nodal lines of the eigenfunctions. See Berdnikov \cite{Be} for more details.

\section{Preliminaries}

\subsection{A generalized Morse lemma} 

Let $\mathcal{N}$ be the set of open neighborhoods of $0\in \mathbb{R}^n$, $n\ge 1$, and set $\mathbb{N}_0^n=\mathbb{N}\cup\{0\}$. In what follows,  for $\alpha=(\alpha_1,\ldots ,\alpha_n)\in \mathbb{N}_0^n$, we will use the usual notations
\begin{equation}\label{pd}
\partial ^\alpha=\partial_{x_1}^\alpha\ldots \partial_{x_n}^{\alpha_n},\quad |\alpha|=\alpha_1+\ldots +\alpha_n.
\end{equation}

\begin{theorem}\label{thm1}
Let $f$ and $p$ be two $C^\infty$ functions defined in $\omega\in \mathcal{N}$.  Let $\ell >2$ be an integer, $\epsilon \in (0,1]$ and assume that the following conditions hold:
\\
$\mathrm{(i)}$ $\partial ^\alpha p(0)=0$ for all $|\alpha|\le \ell -1$ and $|\nabla p|\ge \mathfrak{c}|x|^{\ell-1}$, where $\mathfrak{c}>0$ is a constant,
\\
$\mathrm{(ii)}$ $\partial^\alpha f(x)=\partial^\alpha p(x)+O\left(|x|^{\ell-|\alpha|+\epsilon}\right)$, $0\le |\alpha|\le 2$.
%
\\
Then there exist $\omega_0,\omega_1\in \mathcal{N}$ and $\phi:\omega_0\mapsto \omega_1$ a $C^1$-diffeomorphism such that $\phi(0)=0$ and $f(x)=p(\phi(x))$.
\end{theorem}

\begin{proof}
Define
\[
F(x,\tau)=(1-\tau)f(x)+\tau p(x),\quad (x,\tau)\in \omega \times \mathbb{R}.
\]
Then
\[
\nabla F:=(\nabla_x F,\partial_\tau F)=((1-\tau)\nabla f+\tau\nabla p,-f+p).
\]

Henceforth, $\mathbf{c}>0$ will denote a generic constant independent of $\tau$. 

We have $\nabla F(0,\cdot)=0$ and, since 
\begin{align*}
|\nabla F|&\ge |\nabla_xF|-|\partial_\tau F|
\\
&\ge |\nabla p| -|1-\tau| |\nabla f-\nabla p|-|p-f|
\\
&\ge \mathfrak{c}|x|^{\ell-1}-\left|(1-\tau)O\left(|x|^{\ell-1+\epsilon}\right)\right|-\left|O\left(|x|^{\ell+\epsilon}\right)\right|,
\end{align*}
reducing $\omega$ if necessary, we obtain
\begin{equation}\label{0}
|\nabla F|\ge \mathbf{c}|x|^{\ell-1}\quad \mathrm{in}\; \omega \times [0,2].
\end{equation}

Define in $\omega \times [0,2]$ the vector field
\[
X(x,\tau)=(p(x)-f(x))\frac{\nabla F(x,\tau)}{|\nabla F(x,\tau)|^2},\; x\ne 0,\quad X(0,\tau)=0.
\]
\eqref{0} and $\mathrm{(ii)}$ with $|\alpha|=1$ yield $|X(x,\tau)|\le \mathbf{c}|x|^{1+\epsilon}$, $(x,\tau)\in \omega \times [0,2]$. We claim that $X$ is of class $C^1$ in $\omega \times [0,2]$. The proof, which is quite technical, is provided in Appendix \ref{appA}. 

Let $\xi=(0,\ldots, 0,1)\in \mathbb{R}^{n+1}$. For $x\in \overline{B}(0,2\delta)\subset  \omega$ and $\tau\in [0,2]$, using Cauchy-Schwarz inequality, we obtain
\begin{align*}
\xi \cdot (\xi -X(x,\tau))&=1-\xi \cdot X(x,\tau)
\\
&=1-\frac{(p(x)-f(x))^2}{|\nabla F(x,\tau)|^2}
\\
&\ge 1-\mathbf{c}|x|^{2(1+\epsilon)}.
\end{align*}
Whence
\begin{equation}\label{1}
\xi \cdot (\xi -X(x,\tau))\ge 1-\mathbf{c}\delta^{2(1+\epsilon)}\quad (x,\tau)\in \overline{B}(0,2\delta) \times [0,2].
\end{equation}
We choose $\delta$ sufficiently small in such a way that
\begin{equation}\label{2}
|X(x,\tau)|\le \delta /4,\quad \quad (x,\tau)\in \overline{B}(0,2\delta) \times [0,2].
\end{equation}
Let 
\[
Y(x,\tau)=(Y_1(x,\tau),\ldots ,Y_{n+1}(x,\tau)):=\xi-X(x,\tau). 
\]
By the double inequality $1-\mathbf{c}\delta^{2(1+\epsilon)}\le Y_{n+1}\le 1$ in $\overline{B}(0,\delta) \times [0,2]$, we can reduce $\delta$ once more in order to satisfy
\begin{equation}\label{3}
3/4\le Y_{n+1}\le 1\quad \mathrm{in}\;  \overline{B}(0,2\delta) \times [0,2].
\end{equation} 

Pick $x\in \overline{B}(0,\delta)$, set $U_x:=\overline{B}(x,\delta)\times [0,2]$, $E_0=C([0,2],U_x)$ and 
\[
E:=\{h\in C([0,2],U_x);\; h(0)=(x,0)\}
\]
$E$ is then a closed subspace of $E_0$ which we endow with the natural of $E_0$:
\[
\|h\|_{E_0}=\max_{[0,2]}|h|,\quad h\in E_0.
\]
Consider on $E$ the integral equation
\begin{equation}\label{4}
h(t)=(x,0)+\int_0^tY(h(s))ds,\quad t\in [0,2].
\end{equation}

This integral equation admits a unique solution $h\in E$. To prove this claim, we introduce the mapping $T=(T_1,\ldots, T_n, T_{n+1})=(T',T_{n+1})$ defined as follows
\[
T:h\in E\mapsto Th:Th(t):= (x,0)+\int_0^tY(h(s))ds,\quad t\in [0,2].
\]
Let us first show that $T(E)\subset E$. Clearly, $Th(0)=(x,0)$. On the other hand, we have
\[
T'h(t)=x-\int_0^tX(h(s))ds.
\]
In light of \eqref{2}, we have
\[
|T'h(t)-x|\le \delta ,\quad t\in [0,2],
\]
and \eqref{3} implies that $T_{n+1}h(t)\in [0,2]$. We verify that a power of $T:E\rightarrow E$ is contractive. Precisely, we have
\[
\|T^kh_1-T^kh_2\|_{E_0}\le \frac{M^k}{k!},\quad k\ge1,\; h_1,h_2\in E,
\]
where $M=\max_{U_x}|\nabla Y|$. Therefore, $T$ admits a unique fixed point by the Banach contraction principle. This proves our assertion since $h\in E$ is a solution of the integral equation \eqref{4} if and only if $h$ is a fixed point of $T$.

The solution $h\in E$ of the integral equation \eqref{4} belongs to $C^1([0,2],U_x)$ and it is nothing but the solution of the Cauchy problem
\begin{equation}\label{5}
h'(t)=Y(h(t)), \; t\in [0,2],\quad h(0)=(x,0).
\end{equation}
For convenience,  the unique solution of \eqref{5} will denoted by $h(t,x)$. The well known results on dependence of solutions of Cauchy problems on the initial data show that $h$ is of class $C^1$.

It follows from \eqref{3} that $h_{n+1}([0,2])\supset [0,3/2]$. Hence, there exists $0<t(x)<2$ such that $h_{n+1}(t(x),x)=1$. Since $h_{n+1}$ is increasing, $t(x)$ is unique. According to the implicit function theorem, $x\mapsto t(x)$ is of class $C^1$. Note that since $(0,t)$ is the unique solution of \eqref{5} when $x=0$ we get $t(0)=1$.

We denote by $\phi(x)=(\phi_1(x),\ldots ,\phi_n(x))$ the unique point so that $h(t(x),x)=(\phi(x),1)$. In particular, we have $\phi(0)=0$. 

Our next objective is to show that $\phi$ is $C^1$ diffeomorphism from a  neighborhood of $0$ in $\mathbb{R}^2$ onto another neighborhood of $0$ in $\mathbb{R}^2$. By differentiating $h_{n+1}(t(x),x)=1$, we obtain
\[
\partial_th_{n+1}(t(x),x)\partial_jt(x)+\partial_jh_{n+1}(t(x),x))=0,\quad 1\le j\le n,
\]
and then
\begin{equation}\label{6}
\partial_jt(x)=-\frac{\partial_jh_{n+1}(t(x),x))}{\partial_th_{n+1}(t(x),x)},\quad 1\le j\le n.
\end{equation}
On the other hand, as $\phi_j(x)=h_j(t(x),x)$, $\phi$ is of class $C^1$ and
\[
\partial_k\phi_j(x)=\partial_th_j(t(x),x)\partial_kt(x)+\partial_kh_j(t(x),x),\quad 1\le j,k\le n.
\]
In light of \eqref{6}, we have $\partial_kt(0)=0$ for all $1\le k\le n$. Hence,
\[
\partial_k\phi_j(0)=\partial_kh_j(1,0),\quad 1\le j,k\le n.
\]

Let $1\le j,k\le n$  and set $g_{jk}(s)=\partial_kh_j(s,0)$, $0\le s\le 2$. Then
\begin{equation}\label{7}
g_{jk}'(s)=\partial_k\partial_th_j(s,0),\quad 0\le s\le 2.
\end{equation}
Using $\partial_th_j(t,x)=Y_j(h(t,x))$, we obtain
\begin{align*}
\partial_k\partial_th_j(t,x)&=\sum_{m=1}^{n+1}\partial_mY_j(h(t,x))\partial_kh_m(t,x)
\\
&=-\sum_{m=1}^{n+1}\partial_mX_j(h(t,x))\partial_kh_m(t,x).
\end{align*}
This last identity implies
\[
\partial_k\partial_th_j(t,0)=-\sum_{m=1}^{n+1}\partial_mX_j(h(t,0))\partial_kh_m(t,0)=0
\]
and hence $g_{jk}'=0$. In consequence, we have $g_{jk}(1)=g_{jk}(0)$. From the identity $h_j(0,x)=x$, we get $g_{kj}(1)=\partial_kh_j(1,0)=\partial_k\phi_j(0)=\delta_{kj}$. In other words, $(\partial_k\phi_j(0))=I$, the identity matrix of $\mathbb{R}^n$. By applying local inversion theorem, we conclude that $\phi$ is a $C^1$-diffeomorphism from $\omega_0\in \mathcal{N}$ onto $\omega_1\in \mathcal{N}$.

We complete the proof by showing that $f(x)=p(\phi(x))$ in $\omega_0$. Let $(x,\tau)\in \omega_0\times [0,2]$. From
\[
X(x,\tau)\cdot \nabla F(x,\tau)=p(x)-f(x)=\xi\cdot \nabla F(x,\tau),
\] 
we obtain
\[
(\xi-X(x,\tau))\cdot \nabla F(x,\tau)=Y(x,\tau)\cdot \nabla F(x,\tau)=0.
\]
We use this identity to obtain
\[
\frac{d}{dt}F(h(t,x))=\nabla F(h(t,x))\cdot \frac{d}{dt}h(t,x)=\nabla F(h(t,x))\cdot Y(h(t,x))=0
\]
and then $F(h(t,x))=F(h(0,x))=F(x,0)=f(x)$. On the other hand,  we have $F(h(t(x),x))=F(\phi(x),1)=p(\phi(x))$. That is we have $f(x)=p(\phi(x))$ as expected.
\end{proof}

Theorem \ref{thm1} is due to Cheng \cite{Ch}, based on an idea borrowed from \cite{Ku}. The detailed proof we provide here clarifies some points of the one given in \cite{Ch}.

\subsection{Zeros of two-variable harmonic homogeneous function}

If necessary, we will identify $\mathbb{R}^2$ and $\mathbb{C}$. For $\theta\in [0,\pi[$, we denote $D(\theta)$ the line passing through the origin and directed by the vector $e^{i\theta}$.

\begin{proposition}\label{prop1}
Let $f:\mathbb{R}^2\rightarrow \mathbb{R}$ be a $C^\infty$ harmonic homogeneous function of degree $k\ge 1$ which is non-identically equal to zero. Then there exists $\theta\in [0,\pi[$ such that
\[
Z(f):=\{x\in \mathbb{R}^2;\; f(x)=0\}=\bigcup_{j=0}^{k-1} D(\theta+j\pi/k).
\]
\end{proposition}

\begin{proof}
As $f=f(x,y)$ is homogeneous of degree $k$, we have
\begin{equation}\label{8}
x\partial_xf+y\partial_yf=kf.
\end{equation}
Taking the derivative with respect to $x$ of both sides of \eqref{8}, we obtain
\[
x\partial_x^2f+\partial_xf+y\partial_x\partial_yf=k\partial_xf
\]
and hence
\begin{equation}\label{9}
x^2\partial_x^2f+x\partial_xf+xy\partial_x\partial_yf=kx\partial_xf.
\end{equation}
By interchanging the roles of $x$ and $y$, we get 
\begin{equation}\label{10}
y^2\partial_y^2f+y\partial_yf+xy\partial_x\partial_yf=ky\partial_yf.
\end{equation}
In light of \eqref{8}, by taking the sum side by side of \eqref{9} and \eqref{10} and using that $f$ is harmonic, we obtain 
\begin{equation}\label{11}
y^2\partial_x^2f+x^2\partial_y^2f-kf-2xy\partial_x\partial_yf=-k^2f
\end{equation}
We choose in this identity $(x,y)=(\cos t,\sin t):=\eta(t)$, $t\in [0,2\pi]$, in order to obtain
\begin{align}
 \partial_x^2f(\eta(t))\sin ^2t+&\partial_y^2f(\eta(t))\cos ^2t-kf(\eta(t)) \label{11}
\\
&-2\partial_{xy}^2f(\eta(t))\cos t \sin t=-k^2f(\eta(t)).\nonumber
\end{align}
If  $g(t)=f(\cos t,\sin t)$, then
\[
g'(t)=-\partial_xf(\eta(t))\sin t+\partial_yf(\eta(t))\cos t.
\]
Hence
\begin{align*}
g''(t)&=\partial_x^2f(\eta(t))\sin^2t- \partial_x\partial_y f(\eta((t))\sin t\cos t-\partial_xf(\eta(t))\cos t
\\
&\qquad+\partial_y^2f(\eta(t))\cos^2t- \partial_x\partial_y f(\eta((t))\sin t\cos t-\partial_yf(\eta(t))\sin t
\\
&=\partial_x^2f(\eta(t))\sin^2t+\partial_y^2f(\eta(t))\cos^2t-2\partial_x\partial_y f(\eta((t))\sin t\cos t-kf(\eta(t)),
\end{align*}
which combined with \eqref{11} yields $g''(t)=-k^2g(t)$. That is $g$ is of the form $g(t)=\mathbf{a}\sin (k(t-t_0))$, where $\mathbf{a}$ is a constant and $t_0\in [0,\pi[$. As $f(r\cos t ,r\sin t)=r^kg(t)$ and $f$ is non identically equal to zero, $\mathbf{a}\ne 0$. We complete the proof by noting that $g(t)=0$ if and only if $t=t_0+j\pi/k$ for $j=0,1,\ldots 2k-1$.
\end{proof}

\begin{remark}
{\rm
We have provided a direct elementary proof of the Proposition \ref{prop1}. However, we can use the fact that any harmonic function is by analyticity the infinite sum, in a neighborhood of the origin, of homogeneous harmonic polynomials (e.g. \cite[Exercise 6.16]{Cho}). Therefore, a homogeneous harmonic function is necessarily a homogeneous harmonic polynomial.
}
\end{remark}

\section{Proof of Theorems \ref{mthm1} and \ref{mthm3}}

\begin{proof}[Proof of Theorem \ref{mthm1}]
We first consider the case $k\ge 3$. We proceed by contradiction. Suppose that there exist $f_1,\ldots, f_{2k}$ in $E_k$ which are linearly independent. Define 
\[
f=\sum_{j=1}^{2k} \alpha_jf_j.
\]
 
Pick $x_0\in \Omega$. Due to the translation invariance of the Laplace operator, we assume, without loss of generality, that $x_0=0$. We then choose $(\alpha_1, \ldots ,\alpha_{2k})\in \mathbb{R}^{2k}$  is such a way 
\begin{equation}\label{12}
\partial_y^mf(0)=0,\quad 0\le m\le k-1, \quad \partial_x\partial_y^mf(0)=0,\quad 0\le m\le k-2.
\end{equation}
Since \eqref{12} represents a system of $2k-1$ equations, there exists $(\alpha_1,\ldots ,\alpha_{2k})\in \mathbb{R}^{2k}$, $(\alpha_1,\ldots ,\alpha_{2k})\ne 0$, satisfying \eqref{12}.

By $f\in E_k$, we obtain $\partial_x^2f+\partial_y^2f=Vf-\lambda_k f$ and hence
\begin{equation}\label{13}
\partial_x^2\partial_y^mf+\partial_y^{2+m}f=\sum_{\alpha=0}^m\partial_y^{m-\alpha}V\partial_y^\alpha f-\lambda_k \partial_y^mf,\quad 0\le m\le k-1.
\end{equation}

In light of \eqref{12}, \eqref{13} with $x=0$  implies
\begin{equation}\label{13.3}
\partial_x^2\partial_y^mf(0)=-\partial_y^{2+m}f(0),\quad 0\le m\le k-1.
\end{equation}

Combining \eqref{12} and \eqref{13.3}, we obtain 
\begin{equation}\label{13.5}
\partial_x^2\partial_y^mf(0)=0,\quad 0\le m\le k-3.
\end{equation}

Putting together \eqref{12} and \eqref{13.5}, we get
\begin{equation}\label{13.6}
\partial_x^j\partial_y^mf(0)=0,\quad  0\le j\le 2,\; 0\le m\le( k-1)-j.
\end{equation}

Let $B$ a ball centered at $0$ and contained in $\Omega$. Let $W(y,x)=V(x,y)$ and $g(y,x)=f(y,x)$ $(x,y)\in B$. We have $\partial_x^j\partial_y^m g(y,x)=\partial_x^m\partial_y^j f(y,x)$ for all $(x,y)\in B$ and $j,m\in \mathbb{N}$, and $-\Delta g=\lambda_kg-Wg$ in $B$. We proceed as above to get 
\begin{equation}\label{13.7}
\partial_x^j\partial_y^mf(0)=0,\quad  0\le m\le 2,\; 0\le m\le( k-1)-j.
\end{equation}
A combination of \eqref{13.6} and \eqref{13.7} gives
\begin{equation}\label{13.8}
\partial_x^j\partial_y^mf(0)=0,\quad  0\le j+m\le k-1.
\end{equation}
 
As $f$ is non identically equal to zero, it  cannot admit a zero of infinite order at $0$ (e.g. \cite[Theorem 1.1]{GL}). Let $\ell \ge k$ be the smallest integer  for which there exists two integers $j\ge 0$ and $k\ge 0$ so that $j+m=\ell$ and $\partial_x^j \partial_y^mf(0)\ne 0$.  Taylor's formula  then yields for some $\ell \ge k$
\[
f(x,y)=p(x,y)+r(x,y),\quad (x,y)\in \omega,
\]
where $\omega$ is a neighborhood of $0$ in $\mathbb{R}^2$,
\[
p(x,y)=\sum_{j+m=\ell}\frac{1}{j!m!}\partial_x^j \partial_y^mf(0)x^jy^m
\]
is non identically equal to $0$, $r(x)=O(|(x,y)|^{\ell+1})$ and $\nabla r(x,y)=O(|(x,y)|^\ell)$. We verify that $|\nabla p|\ge \mathbf{c}|(x,y)|^{\ell-1}$, where $\mathbf{c}>0$ is a constant.

On the other hand, we have
\begin{align*}
\partial_x^2p(x,y)&=\sum_{j\ge 2,\; j+m=\ell}\frac{1}{(j-2)!}\frac{1}{m!}\partial_x^j \partial_y^mf(0)x^{j-2}y^m
\\
&=\sum_{j+m=\ell-2}\frac{1}{j!}\frac{1}{m!}\partial_x^{j+2} \partial_y^mf(0)x^jy^m
\end{align*}
and
\[
\partial_y^2p(x,y)=\sum_{j+m=\ell-2}\frac{1}{j!}\frac{1}{m!}\partial_x^j \partial_y^{m+2}f(0)x^jy^m.
\]
Hence
\[
\Delta p(x,y)=\sum_{j+m=\ell-2}\frac{1}{j!}\frac{1}{m!}\left[\partial_x^{j+2} \partial_y^mf(0)+\partial_x^j \partial_y^{m+2}f(0)\right]x^jy^m.
\]
As before, using that
\[
\partial_x^{j+2} \partial_y^mf+\partial_x^j \partial_y^{m+2}f=-\sum_{\alpha=0}^m\partial_y^{m-\alpha}V\partial_y^\alpha f-\lambda_k\partial_x^j \partial_y^mf,
\]
we obtain
\[
\partial_x^{j+2} \partial_y^mf(0)+\partial_x^j \partial_y^{m+2}f(0)=0,\quad j+m=\ell-2.
\]
and then $\Delta p=0$. Therefore, $p$ a is harmonic polynomial of degree $\ell$.

By noting that all the assumptions of Theorem \ref{thm1} are satisfied with $\epsilon=1$, we find a $C^1$-diffeomorphism $\phi$ from a neighborhood $\omega$ of $0$ in $\mathbb{R}^2$ onto a neighborhood $\omega'$ of $0$ in $\mathbb{R}^2$ such that 
\begin{equation}\label{14}
f(x,y)=p(\phi(x,y)),\quad (x,y)\in \omega.
\end{equation}
Now, as $p$ satisfies also the assumptions of Proposition \ref{prop1}, we have for some $\theta\in [0,\pi[$
\[
Z(p):=\{x\in \mathbb{R}^2;\; p(x)=0\}=\bigcup_{j=0}^{\ell-1} D(\theta+j\pi/k).
\]

Define for $1\le j\le \ell-1$
\[
\mathcal{C}_j(x,y)=\phi^{-1}(u,v),\quad (u,v)\in D(\theta+j\pi/k)\cap \omega'.
\]
Let $Z(f)=\{x\in \Omega ;\; f(x)=0\}$. It follows from \eqref{14} that
\[
Z(f)\cap \omega=\bigcup_{j=0}^{\ell-1} \mathcal{C}_j.
\]
Since $\mathcal{C}_j$ intersect only at $0$, we deduce that $f$ admits at least $\ell+1$ nodal domains. According to \cite[Corollary 2.5]{AHH} (Courant's nodal domain theorem for the Schr\"odinger-Dirichlet operator), $f$ must be identically equal to zero. This leads to the expected contradiction. Here, we recall that a nodal domain is a connected component of $\Omega\setminus Z(f)$.

For the case $k=2$ we proceed again by contradiction. To this end, we assume that $m_2\ge 4$. Therefore, there exists $g_1,\ldots ,g_4$ linearly independent. Pick $x_0\in \Omega$ and  let $f_j=g_j-\alpha_jg_4$, $j=1,2,3$. We choose $\alpha_j$ so that $f_j(x_0)=0$, $j=2,3,4$. Clearly, $f_1,f_2,f_3$ are linearly independent. Let $(\beta_1,\beta_2,\beta_3)\ne 0$ chosen in such a way that
\[
\beta_1\nabla f_1(x_0)+\beta_2\nabla f_2(x_0)+\beta_3\nabla f_3(x_0)=0
\]
and set $f=\beta_1f_1+\beta_2f_2+\beta_3f_3$. Therefore, $f\in E_2$ and satisfies $f(x_0)=0$ and $\nabla f(x_0)=0$. Upon making a translation, we assume that $x_0=0$. Let
\[
p(x,y)=ax^2+bxy+cy^2,
\]
where $a=\partial_x^2f(0)/2$, $b=\partial_x\partial_yf(0)$ and $c=\partial_y^2f(0)/2$. Using that
\[
-\partial_y^2f=\partial_x^2f+\lambda_2f-Vf,
\]
we obtain $c=-a$. Hence
\[
p(x,y)=a(x^2-y^2)+bxy.
\]
If $a^2+b^2\ne 0$, then $d^2f(0)$, the second differential of $f$ at $0$, is non-degenerate. By applying the Morse lemma, we deduce that there exists a $C^1$-diffeomporphism $\phi$ from a neighborhood $\omega$ of $0$ in $\mathbb{R}^2$ onto a neighborhood  $\omega'$ of $0$ in $\mathbb{R}^2$ such that $f(x,y)=p(\phi(x,y))$, $(x,y)\in \omega$. We proceed similarly as for the case $k\ge 3$ to derive the expected contradiction. If $a=b=0$, then, again from the unique continuation property, $f$ cannot have a zero of infinite order at $0$. Therefore, we find $\ell \ge 3$, a polynomial $p$ and $C^\infty$ function $r$ satisfying the same assumptions as in the case $k\ge 3$ for which we have $f=p+r$. Once again, we proceed as in the case $k\ge 3$ to obtain the expected contradiction.
\end{proof}


From the preceding proof we deduce the following result. The differential of order $j\ge 1$ of a function $f$ will denoted by $d^jf$.

\begin{theorem}\label{mthm2}
Under the assumptions of Theorem \ref{mthm1}, if $f\in E_k$, $k\ge 2$, satisfies at some $x_0\in \Omega$, $d^jf(x_0)=0$ for $0\le j\le k-1$, then $f=0$.
\end{theorem}

The following lemma will be used to prove Theorem \ref{mthm3}.

\begin{lemma}\label{lem1}
Assume that the assumptions of Theorem \ref{mthm3} are satisfied. Let $k\ge 2$ and set $m=m_k=\mathrm{dim}(E_k)$. Then $E_k$ admits a basis $(f_1,\ldots,f_m)$ so that $f_j$ and $f_j\circ \varrho$ are linearly dependent.
\end{lemma}

\begin{proof}
We first note that if $f\in E_k$, then $f\circ \varrho\in E_k$. Pick $g_1\in E_k\setminus\{0\}$. In the case where $g_1$ and $g_1\circ \varrho$ are linearly dependent, we choose $f_1=g_1$. Otherwise, we choose $f_1=(g_1+g_1\circ \varrho)/2$. As $\varrho^2=I$, where $I$ is the identity of $\mathbb{R}^2$, we verify that $f_1=f_1\circ \varrho$. For the rest of the construction, we proceed by induction. Suppose that we have constructed $(f_1,\ldots ,f_\ell)$, $\ell<m$, so that $f_j$ and $f_j\circ \varrho$ are linearly dependent for $j=0,\ldots \ell$. For simplicity, we assume that $(f_1,\ldots ,f_\ell)$ is an orthonormal system in $L^2(\Omega)$, i.e.
\[
\int_\Omega f_jf_{j'}dx=\delta_{jj'}, \quad 1\le j,j'\le \ell.
\]
By $\ell<m$, we find $g_{\ell+1}\in E_k$ satisfying
\begin{equation}\label{15}
\int_\Omega g_{\ell+1}^2dx=1,\quad \int_\Omega g_{\ell+1}f_j=0,\quad 1\le j\le \ell.
\end{equation}
If $g_{\ell+1}$ and $g_{\ell+1}\circ \varrho$ are linearly dependent, we choose $f_{\ell+1}=g_{\ell+1}$. Otherwise, we choose $f_{\ell+1}=(g_{\ell+1}+g_{\ell+1}\circ \varrho)/2$. We have $\int_\Omega f_{\ell+1}^2dx=1$ and, as $f_j\circ \varrho=\mu_jf_j$, with $\mu_j\in \mathbb{R}$, for all $1\le j\le \ell$, we obtain
\begin{align*}
2\int_\Omega f_{\ell+1}f_jdx&=\int_\Omega \left[g_{\ell+1}dx+ (g_{\ell+1}\circ \varrho)f_j\right]dx
\\
&=\int_\Omega g_{\ell+1}f_jdx+ \int_\Omega g_{\ell+1} (f_j\circ \varrho) dx
\\
&=(1+\mu_j)\int_\Omega g_{\ell+1}f_jdx=0,\quad 1\le j\le \ell \quad \mbox{(by \eqref{15}}).
\end{align*}
This completes the proof.
\end{proof}

\begin{proof}[Proof of Theorem \ref{mthm3}]
We proceed by contradiction by assuming that $m\ge 2k-1$. According to Lemma \ref{lem1}, we can choose in $E_k$ a basis $(f_1,\ldots ,f_m)$ in such a way that $f_j$ and $f_j\circ \varrho$ are linearly dependent for all $1\le j\le m$. Let $f=\alpha_1f_1+\ldots \alpha_mf_m$, where $(\alpha_1,\ldots \alpha_m)\ne 0$ is chosen so that 
\[
f(0)=0,\quad \partial_yf(0)=0,\quad \partial_x^j\partial_y^{p}f(0)=0,\quad j=0,1, \; 2\le p\le k-1.
\]
But $f_j$ and $f_j\circ \varrho$ are linearly dependent. Hence 
\[
\partial_xf(0)=0,\quad \partial_x^j\partial_y^{p}f(0)=0,\quad 0\le j\le k-1, \; p=0,1.
\]
We deduce that 
\[
\partial_y^pf(0)=0,\quad 0\le p\le k-1,\quad \partial_x\partial_y^{p}f(0)=0,\quad 0\le j\le k-1.
\]
This is \eqref{12} in the proof of Theorem \eqref{mthm1}. The expected contradiction follows by mimicking the proof of Theorem \ref{mthm1} starting from \eqref{12}.
\end{proof}

\section*{Acknowledgement} 

I would like to thank Vladimir Bobkov for drawing my attention to the fact that the assumptions of \cite[Lemma 2.4]{Ch} are not sufficient to guarantee its validity, as well as the reference \cite{Pa}, where the author shows in \cite[Korollar 2.9.4]{Pa} that the above-mentioned lemma is valid under an additional assumption. We emphasize that the assumptions of \cite[Korollar 2.9.4]{Pa} differ slightly from those of Theorem \ref{thm1}.

\appendix

\section{}\label{appA}

The assumptions and notations are those of Theorem \ref{thm1} and its proof. Since $X$ is $C^\infty$-smooth in $(\omega\setminus\{0\})\times [0,2]$, it suffices to show that $X$ is continuously differentiable in $(0,\tau)\in \{0\}\times  [0,2]$. First, we prove that
\[
\lim_{x\rightarrow 0, x\ne 0}\partial_jX_i(x,\tau)=0, \quad t\in [0,2],\; 1\le i,j\le n.
\]

Let $(x,\tau)\in \omega\times [0,2]$ and $0\le |\alpha|\le 2$. As 
\[
\partial ^\alpha F(x,\tau)=(1-\tau) \partial^\alpha f(x)+\tau \partial^\alpha p(x),
\]
we obtain from (ii)
\begin{equation}\label{a1}
\partial ^\alpha F(x,\tau)=\partial^\alpha p(x)+O\left(|x|^{\ell-|\alpha|+\epsilon}\right).
\end{equation}
Here and henceforth, $O(\cdot)$ is uniform in $\tau\in [0,2]$. In light of (i), we have
\[
\partial^\beta (\partial^\alpha p)(0)=0,\quad  |\beta|\le \ell -1-|\alpha|.
\]
Applying Taylor's formula, we derive $\partial^\alpha p(x)=O\left(|x|^{\ell-|\alpha|}\right)$, which, in combination of \eqref{a1}, implies
\begin{equation}\label{a2}
\partial ^\alpha F(x,\tau)=O\left(|x|^{\ell-|\alpha|}\right).
\end{equation}

Let $1\le i,j\le n$ and assume that $x\ne 0$. We decompose $\partial_jX_i$ into three terms:
\[
\partial_jX_i= Z_1+Z_2,
\]
where
\begin{align*}
&Z_1(x,\tau)=(\partial_jp(x)-\partial_jf(x)) \frac{\partial_iF(x,\tau)}{|\nabla F(x,\tau)|^2}+(p(x)-f(x)) \frac{\partial_{ij}F(x,\tau)}{|\nabla F(x,\tau)|^2},
\\
&Z_2(x,\tau)= (p(x)-f(x)) \partial_iF(x,\tau)\partial_j\left(\frac{1}{|\nabla F(x,\tau)|^2}\right),
\end{align*}

Combining $|\nabla F(x,\tau)|\ge \mathbf{c}|x|^{\ell-1}$  and \eqref{a2}, we get
\begin{equation}\label{a3}
Z_1(x,\tau)= O\left(|x|^{\epsilon}\right).
\end{equation}

On the other hand, we have
\[
Z_2=-2(p(x)-f(x)) \partial_iF(x,\tau)\sum_k\frac{\partial_{jk}F(x,\tau)\partial_kF(x,\tau)}{|\nabla F(x,\tau)|^4}.
\]
Proceeding as for $Z_1$, we verify that
\begin{equation}\label{a4}
Z_2=O\left(|x|^\epsilon\right).
\end{equation}
Putting together \eqref{a3} and \eqref{a4}, we get
\[
\partial_jX_i(x,\tau)=O\left(|x|^\epsilon\right).
\]

To complete the proof, we show that
\begin{equation}\label{a5}
\lim_{x\rightarrow 0, x\ne 0}\partial_\tau X_i(x,\tau)=0, \quad \tau\in [0,2],\; 1\le i\le n.
\end{equation}
To this end, we verify that, for all $1\le i\le n$ and $\tau \in [0,2]$,
\begin{align*}
&\partial_\tau X_i(x,\tau)=\frac{(p(x)-f(x))(\partial_ip(x)-\partial_if(x))}{|\nabla F(x,\tau)|^2}
\\
&\hskip 2cm -2\sum_k\frac{(p(x)-f(x))(\partial_kp(x)-\partial_kf(x))\partial_kF(x,\tau)\partial_iF(x,\tau)}{|\nabla F(x,\tau)|^4}.
\end{align*}
Similarly as above, we verify that
\[
\partial_\tau X_i(x,\tau)=O\left(|x|^{2\epsilon +1}\right),\quad 1\le i\le n,
\]
from which \eqref{a5} follows.


\begin{thebibliography}{99}
\bibitem{AHH} 
Ancona, A. ; Helffer, B. ; Hoffmann-Ostenhof, T. : Nodal domain theorems à la Courant. Doc. Math. 9 (2004), 283-299.  

\bibitem{Be} Berdnikov, A. :  Bounds on multiplicities of Laplace operator eigenvalues on surfaces. J. Spectr. Theory 8 (2) (2018), 541-554.


\bibitem{Ch} Cheng, S. Y. :  Eigenfunctions and nodal sets. Comment. Math. Helv. 51 (1) (1976), 43-55.

\bibitem{Cho} Choulli, M. : Applied functional analysis. (Analyse fonctionnelle appliqu\'ee.) (French).
Enseignement SUP-Maths. Les Ulis: EDP Sciences, 2024, viii+ 314 p..

\bibitem{GL} Garofalo, N. ; Lin, F.-H. : Unique continuation for elliptic operators: a geometric-variational approach. Comm. Pure Appl. Math. 40 (3) (1987), 347-366.

\bibitem{HMN} Hoffmann-Ostenhof, T. ; Michor, P. W. ; Nadirashvili, N. : Bounds on the multiplicity of eigenvalues for fixed membranes. Geom. Funct. Anal. 9 (6) (1999), 1169-188.

\bibitem{Ku} Kuo, T. C. : On $C^0$-sufficiency of jets of potential functions. Topology 8 (1969), 167-171.


\bibitem{Pa} Pagani, K. F. : Geometrische Eigenschaften der L\"osungen von quasi-linearen, elliptischen Differentialgleichungen zweiter Ordnung in der Ebene. PHD thesis, Z\"urich, 1990.


\end{thebibliography}
\end{document}